\documentclass{amsart}
\usepackage[utf8]{inputenc}
\usepackage{microtype,mathrsfs,amssymb,mathtools,eucal,anyfontsize,mathabx}

\usepackage{hyperref}
\hypersetup{colorlinks,%
citecolor=black,%
filecolor=black,%
linkcolor=blue,%
urlcolor=black,%
}
\usepackage[capitalize,nameinlink,noabbrev]{cleveref}
\usepackage{xcolor}

\makeatletter
\@tfor\next:=abcdefghijklmnopqrstuvwxyzABCDEFGHIJKLMNOPQRSTUVWXYZ\do{%
 \def\command@factory#1{%
 \expandafter\def\csname cal#1\endcsname{\mathcal{#1}}
 \expandafter\def\csname frak#1\endcsname{\mathfrak{#1}}
 \expandafter\def\csname scr#1\endcsname{\mathscr{#1}}
 \expandafter\def\csname bb#1\endcsname{\mathbb{#1}}
 \expandafter\def\csname rm#1\endcsname{\mathrm{#1}}
 \expandafter\def\csname bf#1\endcsname{\mathbf{#1}}
}
 \expandafter\command@factory\next
}
\makeatother 

\newtheorem{theorem}{Theorem}[section]
\newtheorem{prop}[theorem]{Proposition}
\newtheorem{lemma}[theorem]{Lemma}
\newtheorem{cor}[theorem]{Corollary}

\theoremstyle{definition}

\theoremstyle{remark}

\def\and{\ifhmode\unskip\nobreak\fi\ $\cdot$
}

\title{Relative hyperbolicity of hyperbolic-by-cyclic groups}
\author{Fran\c{c}ois Dahmani}
\address{Fran\c{c}ois Dahmani, Institut Fourier, Univ. Grenoble Alpes, F-38000 Grenoble, France}

\email{francois.dahmani@univ-grenoble-alpes.fr}

\author{Suraj Krishna M S }
\address{Suraj Krishna, School of Mathematics, Tata Institute of Fundamental Research, Mumbai 400005, India}

\email{suraj@math.tifr.res.in}

\begin{document}

\maketitle

\begin{abstract}
Let $G$ be a torsion-free hyperbolic group and $\alpha$ an automorphism of $G$. We show that there exists a canonical collection of subgroups that are polynomially growing under $\alpha$, and that the mapping torus of $G$ by $\alpha$ is hyperbolic relative to the suspensions  of the maximal polynomially growing subgroups under $\alpha$. As a consequence, we obtain a dichotomy for growth: given an automorphism of a torsion-free hyperbolic group, the conjugacy class of an element either grows polynomially under the automorphism, or at least exponentially.
%
\end{abstract}

\section{Introduction}

\subsection{Automorphisms and suspensions}

When one considers an automorphism $\alpha$ of a group $G$, one is confronted with several aspects. Geometric: one has a symmetry of the structure of $G$. Dynamical: one has a transformation of $G$ that one can iterate, and possibly take to a limit. Algebraic: one has a new group $G\rtimes_\alpha \mathbb{Z}$. Of course these three points of view have rich interactions. The geometry of the group $G\rtimes_\alpha \mathbb{Z}$ can for instance witness the geometry or the dynamics of $\alpha$, in certain ways.

The most basic examples of this situation are when $G \cong \bbZ^2$, for which $Aut(G) \cong GL_2(\bbZ)$. 
If $\alpha$ has finite order, then $G\rtimes_\alpha \mathbb{Z}$ has a finite index subgroup isomorphic to $\bbZ^3$.   If $\alpha =  \left(\begin{smallmatrix} 1 & 1\\ 0 & 1 \end{smallmatrix}\right)$, then   $G\rtimes_\alpha \mathbb{Z}$ is isomorphic to the (nilpotent) Heisenberg group $\{ \left(\begin{smallmatrix} 1 & a & c \\ 0 & 1 &b \\ 0 & 0 & 1 \end{smallmatrix}\right), a, b, c\in \bbZ\}$. If  $\alpha =  \left(\begin{smallmatrix} 2 & 1\\ 1 & 1 \end{smallmatrix}\right)$, the semi-direct product $G\rtimes_\alpha \mathbb{Z}$, is solvable, but not virtually nilpotent.

In the case where $G$ is the fundamental group of a hyperbolic surface, Thurston   famously classified its automorphisms. Actually, if $\Sigma_g$ is a closed orientable surface of genus $g\geq 2$, we know, since Baer, Dehn and Nielsen,  a correspondence between outer automorphisms of $\pi_1(\Sigma_g)$  (i.e. automorphisms up to post conjugation by an element of $\pi_1(\Sigma_g)$ )  and mapping classes on $\Sigma_g$.  Thurston  classified the mapping classes, and proved that   $\pi_1(\Sigma_g)\rtimes_\alpha \bbZ$ is the fundamental group of a hyperbolic closed $3$-manifold if, and only if, no positive power of $\alpha$ preserves a non-trivial conjugacy class in $\pi_1(\Sigma_g)$ (\cite{Thurston_hyperbolic_3-manifolds}). This latter property is also characteristic of pseudo-Anosov mapping classes of $\Sigma_g$. 

In the context of free groups, Brinkmann proved in \cite{Brinkmann} that the same criterion validates a similar conclusion for free groups: an automorphism   $\alpha$ of a free group $F$   produces a word-hyperbolic group $F\rtimes_\alpha \bbZ$ if and only if it is  atoroidal, in the sense that no positive power of $\alpha$ preserves a non-trivial conjugacy class in $F$.

\subsection{Dropping atoroidality: relatively hyperbolic groups}\label{sec;intro_relhyp}

There are also many non-atoroidal automorphisms of free groups, as well as many non-pseudo-Anosov mapping classes of surfaces. While Thurston's approach manages to elucidate the geometry of those automorphisms as well, it took more effort and time to treat the case of free groups. One reason is that the preserved conjugacy classes of elements, and of subgroups,  in free groups can have a more elaborate towering structure.

The correct geometric notion to treat these cases is relative hyperbolicity.

A group $G$ endowed with a conjugacy closed collection of finitely generated subgroups $\calH$ is relatively hyperbolic (or $G$ is relatively hyperbolic with respect to $\calH$) if it acts co-finitely on a hyperbolic  graph $X$, freely on edges,  on which the elliptic subgroups are exactly the  elements of $\calH$,  each fixing a single vertex, and such that at each infinite valence vertex $v$, the angular metric is proper: given an edge $e$ starting at $v$, and $L>0$, only finitely many edges starting also at $v$ are at distance $\leq L$ from $e$ in $X\setminus \{v\}$. See \cite{bowditch}. Relatively hyperbolic groups form a class that behaves nicely with respect to acylindrical amalgamations and HNN extensions \cite{dahmani_combination}. We will use the combination theorem from that paper (recalled as \cref{thm;combination} below) several times. Let us mention a few cases: the amalgamation of two relatively hyperbolic groups over a subgroup that is maximal parabolic in one of them is again relatively hyperbolic with respect to the same conjugacy classes of subgroups. Similarly, the HNN extension of a relatively hyperbolic group over a subgroup with one attaching map that is maximal parabolic is also relatively hyperbolic. When the attaching maps of the HNN extension, or of the amalgamation, are not maximal parabolic, one needs to extend the collection of parabolic subgroups indeed, and we will have to do that. However, one still can describe the relatively hyperbolic structure with \cite{dahmani_combination}.

In surface groups, the situation is more classical. If $\phi$ is a mapping class of a surface $\Sigma_g$ of genus $g\geq 2$, there is a collection of simple closed curves and an exponent $s$ such that $\phi^s$ preserves the collection and the subsurfaces in the complement, and induces on each subsurface, either the identity, or a pseudo-Anosov mapping class. 

Let $\calC$ denote the collection of subgroups of $\pi_1(\Sigma_g)$ consisting of the conjugates of \begin{enumerate}
    \item the fundamental groups of the subsurfaces on which $\phi^s$ induces the identity, and
    \item the cyclic subgroups generated by the invariant loops on whose neighboring subsurfaces $\phi^s$ does not induce the identity. 
\end{enumerate}
The suspension of the surface by such a mapping class  $\phi^s$ is a 3-manifold, with a decomposition as sub-three manifolds glued together along their boundary components that are homeomorphic to tori (the suspensions of the preserved curves). Some pieces are hyperbolic (with boundary components), some pieces are simply trivial bundles of a subsurface over $S^1$. The resulting fundamental group is, by classical combination theorems \cite{dahmani_combination}, as we mentioned,  relatively hyperbolic, with respect to the direct products of the groups in $\calC$ with the infinite cyclic group centralizing them. The reader with an advance of a couple of paragraphs will have recognized that this collection consists of the suspensions of the maximal polynomially (so to say) growing subgroups of $\pi_1(\Sigma_g)$.

Gautero and Lustig \cite{Gautero_Lustig} were the first to formulate the possibility of a relatively hyperbolic structure in the context of arbitrary (non-atoroidal)  free group automorphisms, which is much more delicate. 
Their ideas were completed by the works of \cite{ghosh} and \cite{dahmani_li}, in the form of a theorem: if $F$ is a finitely generated free group and $\alpha$ is an automorphism, then $F\rtimes_\alpha \bbZ$ is relatively hyperbolic with respect to the suspensions of maximal polynomially growing subgroups (see below for a more precise definition). 

This suggests a more complete geometric picture. The aim of this paper is to encompass these situations in a unified statement.

\subsection{Growth} 

Let $G$ be a word hyperbolic group, and $\alpha$ an automorphism of $G$. What can be said of the geometry of $G\rtimes_\alpha \bbZ$? The aim of this paper is to establish the relative hyperbolicity of $G\rtimes_\alpha \bbZ$ in the case where $G$ is torsion-free, with respect to a minimal family of subgroups. It can happen of course that this family of subgroups is $G\rtimes_\alpha \bbZ$ itself, as it happened for free groups, or surface groups. After all, if $\alpha$ is the identity, $G\rtimes_\alpha \bbZ$ is a mere direct product. But nevertheless our result is sharp in the sense that it is  with respect to a minimal family of subgroups. 

In order to state our main result, we need to use the notion of growth under an automorphism (see \cref{def;growth}) for elements in $G$.  First, we describe how  elements of polynomial growth are organized in $G$.

\begin{theorem}\label{theo;intro1} If $G$ is torsion-free hyperbolic, and $\alpha \in Aut(G)$, there is a finite malnormal family of quasiconvex subgroups $\{A_1, \dots, A_r\}$ such that all elements of $A_i$ are polynomially growing under $\alpha$, and every element that is polynomially growing under $\alpha$ is conjugate in one of the $A_i$. \end{theorem}

Achieving this theorem for free groups, or free products, is done by analyzing actions on $\bbR$-trees coming as limits of actions of $G$. While this approach is still a natural one in the context of hyperbolic groups, we had difficulties in this endeavor. Consider an action of a torsion-free hyperbolic group on an $\bbR$-tree with trivial arc stabilizers. Is it true that elliptic subgroups are finitely generated? If the answer to this question is yes, then there is indeed a treatment of this theorem close to that of free groups or free products. Unfortunately we could not decide this question-- although one could be led to think that there can be an argument towards quasiconvexity directly--  so we treated it differently. Our argument is an induction on the Kurosh rank of $G$. Recall that if the Grushko decomposition of $G$ is $G=H_1*\dots * H_k*F_r$ with $H_i$ freely indecomposable, and $F_r$ free of rank $r$, the Kurosh rank of $G$ is $k+r$. The induction step involves treating the case of a free factor system on which $\alpha$ is fully irreducible (no power preserves a larger free factor than those in the system), and the two special cases of a free product of two factors preserved by $\alpha$, and (the more delicate one) of a free HNN extension of an invariant factor. 

\subsection{Main result}

In order to state our main result, we say that, if $G$ is a group, $\alpha$ an automorphism,  and $A$ a subgroup whose conjugacy class is preserved by some positive power $\alpha^s$   of $\alpha$, we consider $s_0$ the minimal such positive exponent, and write $\alpha^{s_0} (A)= h^{-1}Ah$. Then we say that the suspension of $A$ by $\alpha$ in the semidirect product $G\rtimes_\alpha \langle t \rangle$ is the group $A\rtimes_{ad_h \circ \alpha^{s_0}} \langle t^{s_0} h^{-1}\rangle $. 

\begin{theorem}\label{theo;main_intro} If $G$ is torsion-free hyperbolic, and $\alpha \in Aut(G)$, and $\{A_1, \dots, A_r\}$ is a maximal malnormal family of maximal subgroups whose elements are  polynomially growing under $\alpha$, then $G\rtimes_\alpha \bbZ$ is a relatively hyperbolic group with respect to the conjugates of the suspensions of the $A_i$ by $\alpha$.   \end{theorem}

Again, the proof is based on an induction on the Kurosh rank of $G$. The one-ended case, perhaps more familiar to most specialists of hyperbolic groups, is essentially already known, through a standard study of the JSJ decomposition of $G$. The general case was already approached in the relative sense in \cite{dahmani_li}, in the sense that relative hyperbolicity was established relative to polynomially Grushko-growing subgroups: those whose elements have the displacement of the conjugacy classes of their iterate images by $\alpha^n$  grow polynomially when measured in a Grushko tree for $G$. Induction and telescopy of relative hyperbolicity allows to treat some cases, but as in \cref{theo;intro1}, the low complexity withholds some difficulties. We treat separately the case of $G=H_1*H_2$, and $G=H*_{\{1\}}$, both of which are always entirely polynomially growing when measured in their Bass-Serre trees. Whereas in the first case $G=H_1*H_2$, we may assume that $\alpha$ preserves both $H_1$ and $H_2$, in the second case  $G=H*_{\{1\}}$  the automorphism $\alpha$ preserves $H$ but not necessarily the stable letter of the HNN extension.  Special care is given to this case.

Finally, our result is sharp.   If $G\rtimes_\alpha \bbZ $ is relatively hyperbolic,  any conjugacy class that is  polynomially growing (actually sub-exponentially growing) under $\alpha$  must consist of parabolic elements in the relatively hyperbolic structure.  We refer to \cite[Proposition 1.3]{dahmani_msj} for an argument, that we do not reproduce here, based on the exponential divergence of loxodromic elements in relatively hyperbolic groups. We nevertheless mention two corollaries of this consideration, the first of which was also obtained by Coulon, Hilion, Horbez and Levitt. 

\begin{cor} If $\alpha$ is an automorphism of a torsion-free hyperbolic group $G$, then any conjugacy class of elements of $G$ is either at most polynomially growing  under $\alpha$, or at least exponentially growing under $\alpha$. 
\end{cor}

\begin{cor}
Let $\alpha$ be an automorphism of a torsion-free hyperbolic group $G$. Then $G \rtimes_{\alpha} \bbZ$ is hyperbolic if and only if every conjugacy class has exponential growth under $\alpha$.
\end{cor}

We remark that the above is possible only when $G$ is a free product of a finitely generated free group with finitely many (possibly zero)  hyperbolic surface groups.

\subsection{Acknowledgments.}
SKMS was supported by CEFIPRA grant number 5801-1, ``Interactions between dynamical systems, geometry and number theory". Work on the paper was started when SKMS was visiting l'Institut Fourier. He gratefully acknowledges their hospitality. {We are grateful to the referee for several inputs which have improved the paper.}

\section{Decompositions of a hyperbolic group}

\subsection{\texorpdfstring{$G$}{G}-trees and \texorpdfstring{$(G,\calH)$}{(G,H)}-trees}

Let $G$ be a group. A peripheral structure in $G$ is a finite tuple of conjugacy classes of subgroups of $G$. We make the abuse of saying that $H$ is in the peripheral structure $\calH$ if there is a conjugacy class in the tuple $\calH$ that is the conjugacy class of $H$.

A $G$-tree $T$ is a metric tree endowed with an (isometric) action of $G$. It is co-finite if the quotient $G\backslash T$ is finite. It is bipartite if there is a $G$-invariant coloring of vertices in black and white such that no neighbors have the same color. We will write $G_v$ for the stabilizer of a vertex $v$, respectively $G_e$ for the stabilizer of an edge $e$.

If $G$ is endowed with a peripheral structure $\calH$, one says that a $G$-tree $T$ is a $(G,\calH)$-tree if for all subgroups $H$ of $G$ contained in the structure $\calH$, $H$ fixes a point in $T$, and if moreover, any nontrivial stabilizer of a vertex in $T$ is a subgroup in $\calH$. Accordingly, a group $G$, endowed with a peripheral structure $\calH$, has no cyclic splitting relative to that structure, if for all $G$-trees $T$ with cyclic edge stabilizers, in which each subgroup of $\calH$ is elliptic, $G$ has a global fixed point in $T$.

A free factor system for $G$ is a collection of subgroups $\{H_1, \dots, H_m\}$ such that there exists a free subgroup $F$ of $G$ for which $G= H_1*\dots *H_m*F$. In that case, there exists a $G$-tree with trivial edge stabilizers, for which the elliptic subgroups are exactly the subgroups of conjugates of the $H_i$. Conversely, by Bass-Serre theory, any such tree provides, by a correct choice of representatives of vertex stabilizers, a free factor system.

\subsection{Grushko trees}

Recall that Grushko's theorem says that any finitely generated group is the free product of finitely many freely indecomposable subgroups, and of a free group. Moreover it says that, for any two such decompositions, the peripheral structure of the conjugacy classes of the freely indecomposable subgroups differ only by a permutation. One can name this (unordered) peripheral structure the Grushko peripheral structure, and its elements are the Grushko factors. If $G$ is word-hyperbolic, each Grushko factor is itself word hyperbolic, because it is quasiconvex.

A $G$-tree is a Grushko $G$-tree if it is co-finite, with trivial edge stabilizers, and such that for every vertex $v$, either $v$ has trivial stabilizer and valence $\geq 3$, or $v$ has a freely indecomposable non-trivial stabilizer. For readers familiar with Guirardel and Levitt's \cite{GL_out}, a Grushko $G$-tree is a tree in the outer space for free products of the Bass-Serre tree of the Grushko decomposition of $G$.

\subsection{JSJ trees}\label{subsection_jsj_one-ended_complexity}

We now focus on the case of $G$ being a torsion-free hyperbolic group. If $G$ is freely indecomposable, then there is a unique (up to equivariant isometry) bipartite co-finite  $G$-tree $T_{JSJ}$ that satisfies the following conditions: 
\begin{itemize}
 \item each stabilizer of a black vertex is a maximal infinite cyclic subgroup of $G$; 
 \item each stabilizer of a white vertex  is either the fundamental group of a surface with boundary components, for which the adjacent edge subgroups are the conjugates of the boundary component subgroups (which is called a QH vertex), or  has  no cyclic splitting relative to the peripheral structure of the adjacent edge subgroups; 
 \item any edge stabilizer is elliptic in any $G$-tree whose edge stabilizers are cyclic;
 \item all edges have length $1$. 
\end{itemize}
 This tree is called the canonical JSJ tree of $G$. We refer to the abundant literature, and to the reference \cite{GL_JSJ}.

\subsection{Decompositions adapted to an automorphism}
Now, $G$ is a torsion-free hyperbolic group, and $\alpha$ is an automorphism of $G$. We discuss decompositions of $G$ adapted to $\alpha$.
 
\subsubsection{Maximal free factor system for full irreducibility}\label{sec;ffs_fi}
 
Recall that $\{H_1,\dots, H_m\}$ is a free factor system of $G$ if $G$ possesses a free subgroup $F$ (possibly trivial) for which $G=H_1*\dots *H_m*F$.
 
Recall that an automorphism of $G$ preserves a free factor system $\{H_1,\dots, H_m\}$ of $G$ if it sends each $H_i$ on a conjugate of $H_i$.
 
An automorphism $\alpha$ of $G$ is fully irreducible with respect to a preserved free factor system $\{H_1,\dots, H_m\}$ of $G$ if for all $l\geq 1$, if $\alpha^l$ preserves a proper free factor system $\{Y_1,\dots, Y_k\}$ such that each $H_i$ is conjugated into some $Y_j$, then $\{H_1,\dots, H_m\}=\{Y_1,\dots, Y_k\}$.
 
If $G$ is not freely indecomposable, then some positive power of $\alpha$ preserves a (any) Grushko free factor system of $G$. Up to passing to a power (hence to a finite index subgroup in the suspension of $G$ by $\alpha$) we will assume that $\alpha$ preserves the Grushko free factor system. In that case, by \cite[Theorem 8.24]{francaviglia_martino} (see also \cite[Lemma 1.4]{dahmani_li}) there exists a maximal proper preserved free factor system. Let us denote by $\calH_m$ the peripheral structure of the conjugates of these free factors. It follows that $\alpha$ is fully irreducible with respect to $\calH_m$.

 \subsubsection{The pA-tree in the one-ended case} \label{sec;pA}
 
We now assume that $G$ is one-ended, and still torsion-free.
  
Recall that $G$ has a canonical JSJ decomposition (see \cref{subsection_jsj_one-ended_complexity}), on which $\alpha$ induces an automorphism of graphs-of-groups. After taking some power of $\alpha$, one may assume that $\alpha$ preserves the conjugacy class of each vertex group, and that, for each vertex $v$ with elementary or rigid stabilizer, after conjugating $\alpha^k(G_v)$ back on $G_v$ by an element $g_v$ of $G$, ${\rm ad}_{g_v} \circ \alpha^k$  induces an inner automorphism of $G_v$ (as guaranteed by Bestvina-Paulin-Rips-Sela's argument, consider the exposition in  \cite[Prop 3.1]{Dahmani_Guirardel_gafa}).

Now we may also assume that for all vertices $w$ with QH stabilizer, after conjugating $\alpha^k(G_w)$ back on $G_w$ by an element $g_w$ of $G$, ${\rm ad}_{g_w} \circ \alpha^k$ is preserving a decomposition on the underlying surface, inducing the identity on some pieces, and a pseudo-Anosov automorphism on the other pieces. One may then refine $T_{JSJ}$ by blowing up the QH-vertices according to these decompositions, to obtain a new tree $T_{pA}$, on which edge stabilizers are cyclic, that is still preserved by $\alpha^k$, in the sense that $\alpha^k$ induces an automorphism of graphs-of-groups for the quotient graph-of-groups decomposition $G\backslash T_{pA}$. We call as pA-vertices the vertices of the blow up of the QH-vertices in which $\alpha^k$ induces a pseudo-Anosov automorphism. We denote by $V_{pA}$ the set of these vertices in $T_{pA}$.

\subsection{Relative hyperbolicity} Let us collect some results about relatively hyperbolic groups that we will use later in this paper.    Let $G$ be a finitely generated group endowed with a peripheral structure $\calH$ (as defined above), and consider $\Gamma$  a Cayley graph of $G$ with respect to a finite generating set.  Choose conjugacy representatives $H_1, \dots, H_m$ of the elements of $\calH$, and build the cone-off graph $\hat\Gamma$ by adding to $\Gamma$ a vertex for each left coset of $H_i$, and linking it to the elements of its coset.  One says that $G$ is relatively hyperbolic with respect to $\calH$ if $\hat \Gamma$ is hyperbolic and the angular metric at each vertex is proper (as defined in \cref{sec;intro_relhyp}).  We refer to \cite[Section 2]{bowditch} and \cite[Section 3]{hruska}.   A subgroup $H$ is \emph{maximal parabolic} if $[H] \in \calH$, and \emph{parabolic} if it is contained in a maximal parabolic subgroup. We say that a subgroup $K < G$ is \emph{relatively quasiconvex} if there exists a $\kappa > 0$ such that for any geodesic path $\gamma$ in $\hat \Gamma$ between points of $K$, every vertex of $\gamma$ is 
contained in the $\kappa$-neighborhood of $K$ (see \cite[Section 6]{hruska} and \cite[Definition 1.3]{MPW}).

A subgroup $K$ is \emph{full} if for every $H $ such that $[H] \in \calH$, $K \cap H$ is either finite or has finite index in $H$.

The following combination theorem, proved by the first named author, will be used repeatedly in this paper. Recall that a graph-of-groups is \emph{acylindrical} if there exists $k>0$ such that any segment of length $k$ in the Bass-Serre tree of the graph of groups has finite stabilizer.

\begin{theorem}[\cite{dahmani_combination}]\label{thm;combination}
\begin{enumerate}
    \item Let $G$ be the fundamental group of a finite acylindrical graph of relatively hyperbolic groups such that the edge groups are full relatively quasiconvex subgroups of their incident vertex groups. Then $G$ is hyperbolic relative to the set of $G$-conjugates of the maximal parabolic subgroups of the vertex groups.
    \item Let $G_1$ be a relatively hyperbolic group and let $P < G_1$ be maximal parabolic. Let $A$ be a finitely generated group in which $P$ embeds as a subgroup. Then $G=G_1*_PA$ is hyperbolic relative to the set of $G$-conjugates of $A$ and of the maximal parabolic subgroups of $G_1$, except $P$.
    \item Let $P < G_1, G_2$ be a parabolic subgroup of the relatively hyperbolic groups $G_1$ and $G_2$, with $P$ being maximal parabolic in $G_2$. Then $G = G_1 *_P G_2$ is hyperbolic relative to the $G$-conjugates of the maximal parabolic subgroups of $G_1$ and the maximal parabolic subgroups of $G_2$, except $P$.
    \item[(3')] Let $P \cong P'$ be parabolic subgroups of the relatively hyperbolic group $G_1$, with $P$ maximal parabolic and $P$ not conjugated to $P'$. Then the HNN extension $G = G_1*_P$ is hyperbolic relative to the $G$-conjugates of the maximal parabolic subgroups of $G_1$, except $P$.
\end{enumerate}
\end{theorem}

We also record here Dru\c{t}u's theorem on the quasi-isometric invariance of relative hyperbolicity. 

\begin{theorem}[\cite{drutu}] \label{thm;qi_invariance_drutu} Let $G$ be a group hyperbolic relative to $\calH = \{[H_1], \cdots, [H_n]\}$. If a group $G'$ is quasi-isometric to $G$, then $G'$ is hyperbolic relative to $\calH' = \{[H'_1], \cdots, [H'_m]\}$, where each $H'_i$ can be embedded quasi-isometrically in a conjugate of some $H_j$.  
\end{theorem}

The case of automorphisms that permute conjugacy classes of peripheral subgroups is slightly annoying, but can be reduced to the pure case, as follows.

\begin{prop}\label{prop;fi_polynomial_growth} Let $G$ be a torsion-free hyperbolic group and $\alpha$ an automorphism of $G$. Assume that  $G \rtimes_{\alpha^m} \mathbb{Z}$ is hyperbolic relative to the mapping tori by $\alpha^m$ of a family of subgroups $\calH$, that are quasi-convex and malnormal in $G$.   Then $G \rtimes_{\alpha} \mathbb{Z}$ is hyperbolic relative to the mapping tori by $\alpha$ of the family $\calH$.
\end{prop}

Observe that the malnormality condition on elements of $H$ is actually not needed, and follows from the relative hyperbolicity of  $G \rtimes_{\alpha^m} \mathbb{Z}$.    We prove the proposition, mostly following \cite[Lemma 1.21 and Proposition 1.22]{dahmani_li}. 

 \begin{proof}
First  $G \rtimes_{\alpha^m} \mathbb{Z}$ has finite index in $G \rtimes_{\alpha} \mathbb{Z}$,  therefore, by Dru\c{t}u's
theorem on invariance of relative hyperbolicity by quasi-isometry \cite[Thm 5.1]{drutu},    $G \rtimes_{\alpha} \mathbb{Z}$    is
relatively hyperbolic with respect to a collection of subgroups such that each is at bounded
distance from a peripheral subgroup in $G \rtimes_{\alpha^m} \mathbb{Z}$. 

Let $Q$ be a maximal parabolic subgroup of  $G \rtimes_{\alpha} \mathbb{Z}$, and let $P_i= H_i \rtimes \langle t^m g_i\rangle$ be the maximal parabolic subgroup of   $G \rtimes_{\alpha^m} \mathbb{Z}$  that is at bounded distance from it.   It follows that $Q\cap G$ is at bounded distance from $H_i$, hence is quasi-convex in $G$, as $H_i$ is, and has same limit set in $\partial G$ as $H_i$. By properties of peripheral subgroups, both $H_i$ and $Q\cap G$ are malnormal in $G$, and, being moreover quasi-convex and sharing their limit sets,  therefore they are equal: $G\cap Q = H_i$.

Let $h\in Q$. If $h\in G$ we already know that $h\in H_i$. If $h\notin G$, it conjugates $Q\cap G$  in  $G$ (because $G$ is normal) and in $Q$ (because it is in $Q$ itself), hence in  $Q\cap G$. 
 Thus, $(G\cap Q)^h \subset (G\cap Q)$, and iterating this conjugation, $(G\cap Q)^{h^m} \subset (G\cap Q)$, which means $H_i^{h^m} \subset H_i$. However, $h^m\in G \rtimes_{\alpha^m} \mathbb{Z}$, 
  thus, $h^m\in P_i$, and it follows by definition of $P_i$ that $H_i^{h^m} = H_i$, hence $(G\cap Q)^h = (G\cap Q)$ too. Therefore $h$ is in the mapping torus of $Q\cap G$ in $G \rtimes_{\alpha} \mathbb{Z}$, which is the mapping torus of $H_i$ in  $G \rtimes_{\alpha} \mathbb{Z}$.

Conversely, if $h$ is in the mapping torus of $H_i$ in  $G \rtimes_{\alpha} \mathbb{Z}$, it conjugates $Q$ into a group that intersects $Q$ on an infinite subgroup (namely $Q\cap G$, which is $H_i$), therefore it is in the maximal parabolic group containing $H_i$, which is $Q$. \end{proof} 

The following result, Theorem 1.1 of \cite{yang_ext_peripheral}, allows one to extend the collection of maximal parabolic subgroups in a relatively hyperbolic group.
\begin{theorem}[\cite{yang_ext_peripheral}] \label{theo;yang}
Let $G$ be hyperbolic relative to $\calH$. Let $\calH'$ be a conjugacy closed collection of finitely generated groups such that for each $H \in \calH$, there exists $H' \in \calH'$ such that $H \leq H'$. Then $G$ is hyperbolic relative to $\calH'$ if and only if \begin{enumerate}
    \item Each $H' \in \calH'$ is relatively quasiconvex in $(G,\calH)$.
    \item For every $H'_1, H'_2 \in \calH$, and for every $g \in G$, either $gH'_1g^{-1} \cap H'_2$ is finite or $gH'_1g^{-1} = H'_2$.
\end{enumerate}
\end{theorem}

\section{Growth under an automorphism}

\subsection{Definitions, polynomial growth} \label{def;growth}
Let $G$ be a torsion-free hyperbolic group. Let $d_w$ be a word metric (for some chosen generating set). If $g\in G$, one defines $\|g\|_w$ to be the infimum of $d_w( 1, hgh^{-1})$ over $h\in G$. This is an integer, so it is achieved. We will use the notation $|g|$ to designate the word length of $g$. Let $T$ be a metric $G$-tree, one defines $\|g\|_T$ to be the infimum of $d_T( v, gv)$ for $v$ ranging over the vertices of $T$. Again, for every $g$, this infimum is achieved.

In the following $\|\cdot \|$ is either $\| \cdot \|_w$, or $\| \cdot \|_T$ for a $G$-tree $T$ (with trivial edge stabilizers).

Let $\alpha$ be an automorphism of $G$. We say that $g\in G$ has polynomial $\|\cdot\|$-growth under $\alpha$ if there exists a polynomial $P\in \bbZ[X]$ such that $\|\alpha^n (g)\| \leq P(n)$.

\begin{lemma}\label{lem;H-growth}
If $T_1$ and $T_2$ are two $(G, \calH)$-trees such that their elliptic subgroups are exactly the subgroups in $\calH$, then for any automorphism $\alpha$ preserving $\calH$, any $g\in G$ has polynomial $\|\cdot\|_{T_1}$-growth under $\alpha$ if and only if it has polynomial $\|\cdot\|_{T_2}$-growth under $\alpha$. \end{lemma}

The lemma allows one to talk about polynomial $\calH$-growth whenever $\calH$ is a free factor system, since this is independent of the choice of the $G$-tree provided its elliptic subgroups are exactly the subgroups in the collection $\calH$.

\begin{proof} We may assume that the trees are minimal since the infimum of displacement will be realized in the minimal invariant subtree. Observe that one has an equivariant quasi-isometry from $T_1$ to $T_2$ (this can be worked out by changing the generating set of the graph of groups). The desired result easily follows. \end{proof}

\subsection{Polynomially growing subgroups}

We say that a subgroup $H$ of $G$ has polynomial $\|\cdot\|$-growth under $\alpha$ if all its elements $h\in H$ have polynomial $\|\cdot\|$-growth under $\alpha$ (we stress that this property depends only on conjugacy classes of the elements $\alpha^n(h)$).

We say that a subgroup $H$ of $G$ is maximal polynomially $\|\cdot\|$-growing under $\alpha$ if it has  polynomial $\|\cdot\|$-growth under $\alpha$ and is maximal, with respect to inclusion, among subgroups with this property.

We say that two subgroups $H_1$ and $H_2$ are twinned by $\alpha$ if there is an integer $n\geq 1$ and an element $g$ such that $\alpha^n(H_1) = g^{-1}H_1g$  and $ \alpha^n(H_2) = g^{-1}H_2 g$. 

We say that a family $\{H_1, \dots H_k\}$ of subgroups is a malnormal family if whenever $A=g_a H_i g_a^{-1}$   and   $B= g_b H_j g_b^{-1}$, if $A\cap B\neq \{1\}$, one has $i=j$ and  $g_b^{-1} g_a \in H_i$.

The following theorem, when restricted to the case of free groups,  is due to Levitt \cite{levitt_growth_types}, complementing  Levitt's work with Lustig. It  was elaborated on in \cite{dahmani_li} in the context  of free products. We extend these works to torsion-free hyperbolic groups, to obtain the following.

\begin{theorem}\label{thm;poly_subgps}
Let $G$ be a torsion-free hyperbolic group, and let $\alpha$ be an automorphism of $G$. Let $\|\cdot \|$ denote either $\| \cdot \|_w$, or $\| \cdot \|_T$ for a $G$-tree with trivial edge stabilizers. 
There exists a finite malnormal family $H_1, \dots, H_k$ of quasiconvex subgroups of $G$ such that:
\begin{itemize} 
 \item for each element $h$ of $H_i$, $h$ has polynomial $\|\cdot\|$-growth under $\alpha$; 
 \item conversely, if $h$ has polynomial $\|\cdot\|$-growth under $\alpha$, then there is $g\in G$ and $i$ such that $ghg^{-1}\in H_i$;
\item if $H_i$ is preserved by $\alpha$, then  
for all $h\in H_i$ the sequence of word lengths $|(\alpha^n (h))|$ is  
bounded above by a polynomial in $n$; 
 \item if $H_i, H_j$  are twinned by $\alpha$, then $i=j$. 
\end{itemize}
\end{theorem}
 
Observe, as it will be useful, that these properties imply that there is at most one subgroup that is conjugate to one of the $H_i$ and that is preserved by $\alpha$. Indeed, assume that there are two, $A_1$ and $A_2$,  take non-trivial elements in both $a_i\in A_i$, the group $\langle a_1, a_2 \rangle $ that they generate consists entirely of elements $g$ such that the sequence of word lengths $|(\alpha^n (g))|$  is polynomially growing, therefore, it is a subgroup of a certain conjugate $B$ of one of the $H_j$. But $B$ intersects both $A_i$, and is different from one of them,  which contradicts the malnormality of the family $H_1, \dots, H_k$. Similarly, one can show that if the sequence of word lengths $|(\alpha^n (h))|$  is polynomially growing then $h$ is in a polynomially growing subgroup that is preserved by $\alpha$ (the element $h$ belongs  in a polynomially growing subgroup $A$, and its image $\alpha(h)$ is also contained in a polynomially growing $B$, and the subgroup $\langle h, \alpha(h) \rangle$ also, thus by malnormality, $B=A$).

We will prove this result in the next subsection, by an induction argument. We only indicate here the vocabulary and some simple facts.
 
We call the (unordered) tuple of conjugacy classes of the subgroups $H_i$, the polynomially growing peripheral structure under $\alpha$. If $H$ is in the peripheral structure, we say that $H$ is maximal polynomially growing under $\alpha$. Of course the family (even the cardinality $k$) depends on the choice of $\|\cdot \|$ among $\| \cdot \|_w$, or $\| \cdot \|_T$.

If $d_w$ and $d_{w'}$ are two word metrics of $G$, the (unordered) polynomially growing peripheral structures under $\alpha$ for $\|\cdot\|_w$ and for $\|\cdot\|_{w'}$ are equal. Similarly, if $T_1$ and $T_2$ are two Grushko $G$-trees, the polynomially growing peripheral structures under $\alpha$ for $\|\cdot\|_{T_1}$ and for $\|\cdot\|_{T_2}$ are equal.

If $T$ is a $G$-tree, and $G_v$ a vertex stabilizer in $T$, then $G_v$ is a subgroup of one of the maximal polynomially growing subgroups under $\alpha$ for $\|\cdot\|_T$. In particular, if $G$ is freely indecomposable, $G$ is the unique maximal polynomially growing subgroup under $\alpha$ for $\|\cdot \|_T$, when $T$ is a Grushko tree.

Any maximal polynomially growing subgroup under $\alpha$ for $\| \cdot \|_w$ is a subgroup of a maximal polynomially growing subgroup under $\alpha$ for $\| \cdot \|_T$ (for any $G$-tree $T$).

\subsection{Proof of \texorpdfstring{\cref{thm;poly_subgps}}{Polynomial-subgroups Theorem}}
 
Let us prove the theorem by induction on the Kurosh rank of $G$ (endowed with the Grushko free factor system). Recall that if $G= H_1*\dots *H_k *F_r$, with  $F_r$ free of rank $r$, then the Kurosh rank of the free factor system $\{H_1,\dots, H_k\}$ is $k+r$. In the Grushko free factor system, all $H_i$ are freely indecomposable.

\subsubsection{Kurosh rank $1$}

If the Kurosh rank is $1$, then either $G$ is cyclic and there is nothing to prove, or it is freely indecomposable. In the latter case, the result appears in \cite{Gautero_Lustig_one-ended}. Let us discuss  a possible way to cover it. We consider the pA-tree $T_{pA}$ of \cref{sec;pA}. Each component of the complement of the pA-vertices is a subtree of $T_{pA}$ whose leaves (edges whose one end has valence $1$) are adjacent to pA-vertices. {We call these components collapsible components.} {We consider the following $G$-tree $\bar T_{pA}$: its vertices are the pA-vertices of $T_{pA}$, together with one vertex for each collapsible component,   and there is an edge  between a pA-vertex $v$ and a collapsible component vertex if and only if $v$ is a leaf of the component.} It is straightforward that this is a tree endowed with a $G$-action.   

Recall that $\alpha$ induces an automorphism of the tree $T_{pA}$, hence it induces also an automorphism of $\bar T_{pA}$: there is a map $\bar T_{pA}\to \bar T_{pA}$  equivariant for the original action of $G$ precomposed by $\alpha$, that is a tree automorphism, thus preserving adjacency and length.

For each vertex of $\bar T_{pA}$, its stabilizer in $G$  is either the stabilizer of a pA-vertex in $T_{pA}$ or is a one-ended hyperbolic group (relative to its peripheral structure)  whose JSJ tree has no pA-vertex for the automorphism $\alpha$ (its minimal subtree in $T_{pA}$ is a collapsible sub-tree).  

We claim that the maximal polynomially growing subgroups are the stabilizers of the vertices that are not pA-vertices.  It is clear that each such group is polynomially growing. Take an element $g$ not conjugate to one of them. If it is elliptic in $\bar T_{pA}$, it is in the stabilizer of a pA-vertex, hence its conjugacy class  grows exponentially fast. If it is not elliptic, it is loxodromic.

Observe now that $G\backslash \bar T_{pA}$, the associated graph-of-groups splitting of $G$, is bipartite with one class of vertices being the $pA$-vertices. If $g$ is loxodromic in   $\bar T_{pA}$, the cyclically reduced normal form of its conjugacy class in the graph-of-groups is of the form $g_0 e_0 g_1 e_1 \dots g_ke_k$ with $k\geq 1$,  $g_i$ an element of a vertex group (possibly trivial), and $e_i$ the edge generator between the vertices of $g_i$ and $g_{i+1}$.  After cyclic permutation, we may assume that $g_1$ is an element of a $pA$-vertex group.   If $e_0 = \bar e_1$, by definition of normal form, $g_1$ is not in the edge group of $e_0$.   After post-conjugation, we may assume that (some power of) $\alpha$ preserves the group $G_1$  of $g_1$ and the adjacent edge group of $e_0$.   Since moreover, $\alpha$ induces a pseudo-Anosov automorphism of the vertex group $G_1$,    under the iteration of the automorphism $\alpha$, the  conjugacy class $[\alpha^r(g_1)]$
 has to eventually  grow exponentially fast. Finally, observing that the images by $\alpha$ of the given cyclically reduced normal form are cyclically reduced normal forms of the images (with $\alpha(e_0) = \alpha(e_1)^{-1} = he_0$ for some peripheral $h \in G_1$, hence not growing under $\alpha$),    we obtain that the length of $[\alpha^r(g)]$ is exponential in $r$. If  now $e_0 \neq  \bar e_1$, one must take into account that possibly $g_1$ can be trivial. However, then again after post-conjugation, we may assume that (some power of) $\alpha$ preserves the group $G_1$ of the end vertex of $e_0$,  containing $g_1$,  and also the edge group of $e_0$ (hence the group $G_0$, containing $g_0$, too). However, in that case, the pseudo-Anosov automorphism of the underlying surface of the group $G_1$ has to send the boundary component subgroup  corresponding to $e_1$ to a conjugate by a non-peripheral element $h_1 \in G_1$ (otherwise it would preserve a pair of pants containing both boundary components). It follows that $\alpha^r (g_0 e_0 g_1 e_1 ) = \alpha^r(g_0) e_0 \alpha^r(g_1)  \alpha^{r-1} (h_1)  \dots  \alpha(h_1) h_1    e_1 $.   The sequence  $\gamma_r = \alpha^r(g_1)  \alpha^{r-1} (h_1)  \dots  \alpha(h_1) h_1$ in $G_1$ corresponds to the sequence of iterates of the pseudo-Anosov automorphism applied to an arc joining two boundary components of the surface. By property of the pseudo-Anosov map, the elements $\gamma_r$ need to be all different when $r$ varies, as otherwise a power of the  pseudo-Anosov map would preserve a non-peripheral arc.    If the growth exponent of the pseudo-Anosov map is $\lambda >1$, then for $1<\lambda'<\lambda $, eventually (for  $r$ large enough),  $|\gamma_{r+1} |\geq \lambda'|\gamma_r|  - |h_1| $, and also  $|\gamma_{r} |> |h_1|/(\lambda'-1) $,    which ensures an exponential growth of the sequence $|\gamma_r|$. 

If now two vertex groups of polynomial growth are twinned by $\alpha$, let $v_1,v_2$ be the two considered vertices, and assume that they are different. Thus, there exists a pA-vertex $w$ in the segment $[v_1, v_2]$, and let $e_0, e_1$ be the two edges on this segment issued from $w$. After possible post-conjugation, and after taking a possible power of $\alpha$, we have (denoting also by $\alpha$ the induced automorphism of $\bar T_{pA}$),  $\alpha([v_1, v_2] ) = [v_1,v_2] $ and  since $\alpha$ is a tree-automorphism, $\alpha(w)= w$, and $\alpha(e_0) = e_0, \alpha (e_1)=e_1$.    According to the previous discussion,  this means $e_0=e_1$, which contradicts the assumption that $v_1\neq v_2$. 
 
 \subsubsection{Full irreducibility}
 
Assume the Kurosh rank is higher. Then consider   a maximal free factor system $\calH$ for which $\alpha$ is fully irreducible (see \cref{sec;ffs_fi}), and $T$ an associated Bass-Serre tree. 
 
Assume that in the tree $T$ there are at least $2$ orbits of edges. Then, as proved in \cite[Prop 1.13]{dahmani_li}, there is a  malnormal collection of maximal polynomially $\|\cdot \|_T$-growing subgroups, each having lower Kurosh rank, and each being a hyperbolic group, relatively quasiconvex with respect to $\calH$,   that is finite up to conjugacy, and that contains all elements of  polynomial $\|\cdot \|_T$-growth, and that finally satisfy the no-twinning condition. 
 
We may apply the induction assumption to them and obtain a deeper conjugacy-finite collection of subgroups of polynomial $\|\cdot\|_w$-growth that satisfy the desired properties. Since  any element of polynomial $\|\cdot \|_w$-growth has to be of polynomial $\|\cdot \|_T$-growth, we have the result for $G$.

Assume that in the tree $T$ there is one orbit of edges. This means that the free factor system $\calH$ either decomposes $G$ as $G= H_1 * H_2$, or as $G= H*_{\{1\}} \simeq H*\bbZ$.

\subsubsection{Case of a free product}\label{sec;free_product}

In the first case, $G= H_1 * H_2$. We will give a description of the polynomially growing subgroups, in terms of those in $H_1$ and $H_2$.  We may assume, up to taking $\alpha^2$  that $\alpha$ preserves the conjugacy class of $H_1$, and that of $H_2$. After postconjugating (which does not affect the polynomial growth of elements) we may assume that $\alpha(H_1)= H_1$, and therefore, since $\alpha(H_1)\cup \alpha(H_2) $ must generate $G$, $\alpha(H_2)$ is conjugate to $H_2$ by an element of $H_1$.  So, after postconjugating again  we may assume that $\alpha$ preserves both $H_i$. We also may apply the induction hypothesis to $H_1$ and $H_2$. The difference with the previous case is that it is possible that elements of $G$ are of polynomial $\|\cdot \|_w$-growth without being conjugate to either $H_1$ or $H_2$. However, to be of polynomial $\|\cdot \|_w$-growth their normal forms have to consist only of $|\cdot|$-polynomially growing elements of $H_1$ and $H_2$ in subgroups that are preserved by $\alpha$ in $G$  (the normal form structure is preserved by $\alpha$). 
Observe that there can be only two such groups (any two in the same $H_i$ are not twinned by the induction assumption).  Our collection of subgroups is therefore 
\begin{itemize} 
\item $H_1$-conjugacy representatives of subgroups $A<H_1$,  that are   maximal polynomially $\|\cdot \|_w$-growing for $\alpha$, and have no twin in $H_2$,  
\item $H_2$-conjugacy representatives of subgroups $B<H_2$, that are maximal polynomially $\|\cdot \|_w$-growing for $\alpha$, and have no twin in $H_1$, 
\item and the $G$-conjugacy representatives of the subgroup $A*B$, with $A<H_1$ and $B<H_2$, both preserved by $\alpha$, both respectively maximal polynomially $| \cdot |$-growing for $\alpha$ (there is a unique such pair $(A,B)$, as noted above). 
\end{itemize}    
The desired properties in the statement are easily verified.

\subsubsection{Case of a free HNN extension}\label{sec;freehnn}

In the second case, $G= H * \langle s \rangle $. We will give a description of the polynomially growing subgroups, in terms of those in $H$, but the behavior of $s$ makes it more involved than earlier. We may as above assume that $\alpha$ preserves $H$ and that $\alpha(s) = sh^{-1}$ with $h\in H$. Again, we apply the induction hypothesis for $H$, and obtain a collection of maximal polynomially $\|\cdot \|$-growing subgroups in $H$ that are quasiconvex in $G$. There might be more elements that are polynomially growing.

Observe that by the no-twinning property for $H$, there is a unique maximal polynomially $\|\cdot \|_w$-growing subgroup of $H$ that is fixed by $\alpha$ (however, it can be trivial!). Let $A_0$ be this subgroup of $H$.

\begin{lemma}\label{lem;growth_sbs-1}  For $b\in H$, the sequence  $(\alpha^m(sbs^{-1}))_m$ is a sequence of elements whose word length in $G$ is bounded by a polynomial, if and only if $b$ is in the maximal (possibly trivial)  polynomially growing subgroup fixed by conjugation by $th$. If $B_0$ denotes this latter group, $B_0^{s^{-1}}$ is normalized by $t$.  \end{lemma}  
\begin{proof} Assume that $b\in B_0$. $\alpha(sbs^{-1}) = sh^{-1} t^{-1} b th s^{-1}$, which is $s b_1 s^{-1}$ for  $ b_1$ in $B_0$, image of $b$ by a postconjugation of $\alpha$. Iterating the use of $\alpha$, we get a sequence $b_i$ whose length grows polynomially by assumption on $B_0$. Conversely, if $(\alpha^m(sbs^{-1}))_m$ grows at most polynomially under $\alpha$, then $b$ belongs to a maximal $\|\cdot \|_w$-polynomially growing subgroup of $H$ under $\alpha$. The element $b^{th}$ appears in the identity $\alpha^m (sbs^{-1}) = \alpha^{m-1} (sb^{th} s^{-1})$, therefore, it too belongs to the set of elements $\beta$ for which $(\alpha^m(s\beta s^{-1}))_m$ grows at most polynomially in word length. Because we are considering the word length, the group generated by $b$ and  $b^{th}$ also belong to this set, and we thus find that   $b$ and  $b^{th}$ belong to the same $\|\cdot \|_w$-polynomially growing subgroup. This reveals that this subgroup is normalized by $th$.      To conclude, $B_0^{s^{-1} t} = B_0^{t hs^{-1}} = B_0^{s^{-1}}$, hence the last assertion. \end{proof}

We introduce the following notation that we will use in the next computations: if $k\in H$, we set $k' =k^{-1} h^{-1} \alpha(k) $.  We also introduce $K$, the set of solutions of the membership equation $k'\in A_0$, that is, $K=\{k\in H, k^{-1} h^{-1} \alpha(k)\in A_0 \}$. It could be empty.

\begin{lemma}\label{lem;type2} If $k\in H$, the element $sk$ has polynomial $|\cdot |$-growth if and only if $ k \in K $. \end{lemma} 
 
\begin{proof} Observe that if $k\in H$,  then the image of    $sk$ by  $\alpha$   is $\alpha(sk) =  sh^{-1} \alpha(k) = sk k'   $.  Thus the image of $sk$ by $\alpha^{m}$ is  $sk k' \alpha(k')\dots \alpha^{m-1}(k')$. Since  $k k' \alpha(k')\dots \alpha^{m-1}(k') \in H$, we have a cyclically reduced form in the free product, and therefore $sk$ has polynomial $|\cdot |$-growth if and only if the sequence of word length of the elements $k k' \alpha(k')\dots \alpha^{m-1}(k') \in H$ is bounded above by a polynomial.  If $ k    \in K$, then $k'\in A_0$ and its images in $H$ by $\alpha^m$ grow polynomially in the word metric (because $A_0$ is preserved by $\alpha$). Then $sk$ has polynomial $|\cdot |$-growth. For the converse, assume that there is a polynomial $P$ such that the sequence of elements $k k' \alpha(k')\dots \alpha^m(k')$  has word length bounded by  $P(m)$. Then for $m>1$,  $k k' \alpha(k')\dots \alpha^m(k') \times\alpha^{m+1}(k') $ has, even after cancellation,  word length at least $ |\alpha^{m+1}(k') | -P(m)$, and at most $P(m+1)$.  So we have $P(m+1) +P(m)  \geq |\alpha^{m+1}(k') | $. Thus the word length of  $\alpha^{m}(k')$ grows at most polynomially in $m$, which means that  $k' \in A_0$, and that $k\in K$. 
\end{proof}

In particular, for $k=1$, $s$ has polynomial $\|\cdot\|_w$-growth if and only if $1 \in K$, which holds if and only if $h \in A_0$. 

\begin{lemma}\label{lem;Kcoset} If $K\neq \emptyset$, then it contains exactly one left coset of $A_0$.\end{lemma} \begin{proof} Consider $k_1, k_2 \in K$, the two sequences of word lengths of elements $|\alpha^m (sk_i)|$ grow polynomially, hence this is true also for the sequence $|\alpha^m ((sk_1)^{-1}sk_2)|$. In other words, $|\alpha^m (k_1^{-1}k_2)|$  grow polynomially, and, since $k_1^{-1}k_2 \in H$,   this means that $k_1^{-1}k_2$ is in $A_0$.\end{proof} 

We thus choose $k_0\in K$ if the latter is non-empty, and the group $\langle A_0, \{sk, k\in K\}\rangle$ is equal to $\langle A_0, sk_0\rangle$ (in particular it is finitely generated). We will show below that this subgroup is maximal for polynomial $\|\cdot\|_w$-growth.

Recall that we already considered the images of elements of the form $sbs^{-1}$ in \cref{lem;growth_sbs-1}. We refine this when $K$ is non-empty.
 
\begin{lemma}\label{lem;type1} If $K\neq \emptyset$,  and $b\in H$, the sequence  $(\alpha^m(sbs^{-1}))_m$ is a sequence of elements whose word length in $G$ is bounded by a polynomial, if and only if, for all $k\in K$, the element $a= k^{-1} b k$ is in $A_0$.  \end{lemma}
 
\begin{proof} Consider the element $sbs^{-1}$, and assume that there is $k\in K$.  Write $b= k a k^{-1}$ (of course $a\in H$). Then the image of  $sbs^{-1}$ by   $\alpha^m$ is  \begin{align*} \alpha^m (sbs^{-1} ) =  & sk  \times \left( k' \alpha(k')\dots \alpha^{m-1}(k')\right) \times  \alpha^m(a) \times \\  & \qquad   \times \left(    \alpha^{m-1}((k')^{-1}) \dots  \alpha ((k')^{-1}) (k')^{-1}\right) \times ( k^{-1} s^{-1}).\end{align*}  
 { Since $k\in K$, $k'\in A_0$. We therefore obtain that $\left( k' \alpha(k')\dots \alpha^{m-1}(k')\right)$ and  $\left(    \alpha^{m-1}((k')^{-1}) \dots  \alpha ((k')^{-1}) (k')^{-1}\right)$ define elements whose word lengths have polynomial growth.  Thus the sequence  $(\alpha^m(sbs^{-1}))_m$ has polynomial word length growth if and only if the sequence $\alpha^m(a)$ has. This is characterized by $a\in A_0$.}  \end{proof}

\begin{cor}\label{coro;A_0B_0} If $k\in K \neq \emptyset$, and if $B_0$ is the maximal polynomially growing subgroup fixed by conjugation by $th$ (possibly trivial), then $B_0^k =A_0$. \end{cor} 
\begin{proof} \cref{lem;growth_sbs-1} ensures that $B_0^{s^{-1}}$ is preserved by  conjugation by $t$. We may apply the previous lemma for $b\in B_0$. One obtains that $k^{-1} bk \in A_0$. Conversely, take an element $a$ in $A_0$, and consider $ ka k^{-1} $, and apply the conjugation by $th$. Using that $k\in K$, one obtains that this is an element of $kA_0k^{-1}$.  Therefore $ kA_0 k^{-1} <B_0 $, hence the equality. \end{proof}

Consider, more generally, when $g$ is not a power of $s$, the expression of $g$ as normal form \[ g = s^{\epsilon_1} a_1 s^{\epsilon_2} a_2 \dots s^{\epsilon_r} a_r   \]
with $\epsilon_i = \pm 1$ and $a_i \in H$ (possibly trivial).  Define syllables as the following four possibilities: $sas^{-1}, sa, as^{-1}, a $ (of type 1,2,3,4 respectively) .    Observe that $g$ has a unique decomposition into syllables with the condition that   
\begin{itemize}
    \item a syllable of type 1 is followed only by type 3 and 4,
\item a syllable of type 2 is followed only by type 1 and 2,
\item a syllable of type 3 is followed only by type 3 and 4,
\item a syllable of type 4 is followed only by type 1 and 2.
\end{itemize}

We now define admissible syllables. A type 1 syllable is admissible if $a\in B_0$, a type 2 is if $a\in kA_0$, a type 3 is if  $a\in A_0k^{-1}$, and a type 4 is if $a\in A_0$.

Observe that so far, by \cref{lem;type2} and \cref{lem;type1},  we have established the following.
\begin{lemma}\label{lem;syllable}
 A syllable is $|\cdot |$-polynomially growing if and only if it is admissible.
\end{lemma}

\begin{lemma}
Let $g = s a_1 s^{\epsilon_2} a_2 \dots s^{\epsilon_r} a_r$, with the expression being reduced, $a_i \in H$, $a_r \neq 1$ and $\epsilon_i = \pm 1$. Then $g$ has polynomial $|\cdot|$-growth if and only if all its syllables are admissible.   This is also equivalent to the membership $g\in \langle A_0, sk_0, B_0^{s^{-1}}\rangle$ if $K\neq \emptyset$, and  $g\in \langle A_0, B_0^{s^{-1}}\rangle$ if $K=\emptyset$. 
\end{lemma}

Note that any $g \in G\setminus H$ can be written in the above form after conjugating and possibly taking $g^{-1}$, provided that $g$ is not a power of $s$.

\begin{proof}
$\alpha$ takes any syllable to a syllable of same type. Therefore by \cref{lem;syllable},   all syllables are admissible if and only if  the element $g$ is polynomially growing in word length. This proves the lemma. 
\end{proof}
 
We thus obtained that the collection of maximal subgroups of polynomial $\|\cdot \|_w$-growth of $H*\langle s\rangle$  are 
\begin{itemize}
    \item the conjugates of those of $H$, except the conjugates of $A_0$, and either
    \item the conjugates of  $A_0 * B_0^{s^{-1}}$, if  the set  $K$ of  solutions of $ k^{-1} h^{-1} \alpha(k) \in A_0 $ is  $K=\emptyset$, or
    \item the conjugates of $\langle A_0, \{sk, k\in K\}\rangle$, if $K\neq \emptyset$ (observe that by \cref{coro;A_0B_0}, this subgroup contains $B_0^{s^{-1}}$).
\end{itemize}

It is clear that only $A_0 * B_0^{s^{-1}}$ or $\langle A_0, \{sk, k\in K\}rangle$, as the case may be, among them, is preserved by $\alpha$. 

We argue toward quasiconvexity and malnormality of  $\langle A_0, sk\rangle$, when $k\in K$ (the case of $A_0* B_0^{s^{-1}}$ being similar).    The group $\langle A_0, sk\rangle$ is the free product of $A_0$ by $\langle sk \rangle $, and by free construction it is relatively quasiconvex with respect to $H$ (for instance, use the definition of Mart\'{i}nez-Pedroza and Wise \cite{MPW} in the Serre tree of the free product of $H$ by $\langle s\rangle$). Since its intersection with $H$ is quasiconvex in $H$, it is globally quasiconvex (this can be seen, for example, from \cite[Def. QC-5]{hruska}).  We now argue toward malnormality of  $\langle A_0, sk\rangle$. If two conjugates $\langle A_0, sk \rangle$ and $\gamma\langle A_0,sk\rangle \gamma^{-1}$ intersect non-trivially, a first case is when the intersection contains a non-trivial elliptic element.  So in this case, we can assume it is in $A_0$. Call it $a_0$.  So $a_0 = \gamma b \gamma^{-1}$ for $b$ in $\langle A_0,sk \rangle$. Since $b$ is an elliptic element, write $b = \eta a_1 \eta^{-1} $ with  $\eta \in \langle A_0,sk \rangle$, and $a_1\in A_0$. By malnormality of $H$ in the free product,  the element $\gamma \eta$ is in $H$.  By malnormality of $A_0$ in $H$, the element $\gamma \eta$ is in $A_0$. Since $\eta \in \langle A_0,sk \rangle$, $\gamma$, which is $ \gamma \eta \eta^{-1}$ also has to be in  $\langle A_0,sk\rangle$. The second and last case, is when the intersection of  $ \langle A_0,sk \rangle$ and $\gamma\langle A_0,sk\rangle\gamma^{-1}$ contains an element that is hyperbolic in the free product.  Consider $ \ell_1 = \gamma \ell_2 \gamma^{-1} $ with $\ell_1, \ell_2$ in  $\langle A_0,sk\rangle$, hyperbolic elements. Let $L_1, L_2$ be their axes in the Serre tree of $H*\langle s \rangle$. Up to conjugating $\ell_1$ and $\ell_2$ in $\langle A_0,sk\rangle$, we may assume that both the axes pass through the vertex fixed by $A_0$, and that $\gamma$ fixes this vertex. Since there is a unique $A_0$-orbit of edges adjacent to this vertex in the minimal subtree of $\langle A_0,sk \rangle$, we may assume that the two axes   share a same edge about the vertex fixed by $H$, and that $\gamma$ fixes this edge. But edge stabilizers are trivial, hence, up to multiplication by elements of $\langle A_0,sk\rangle$,  $\gamma$ is trivial, which is what we wanted.

The other properties are easily obtained, and this finishes the proof of \cref{thm;poly_subgps}.

\subsection{Suspensions of the polynomially growing subgroups}

Since the collection of maximal $\|\cdot \|$-polynomially growing subgroups is $\alpha$-invariant, and is finite, we may find $k\geq 1$ such that $\alpha^k$ preserves the conjugacy class of each of them. Specifically, if $Q$ is a maximal polynomial $\|\cdot \|$-growing subgroup of $G$, let $k$ be the smallest positive integer for which there exists $g_q$ satisfying $\alpha^k (Q) = g_q^{-1}Qg_q$. Then the suspension of $Q$ in $G\rtimes_\alpha \langle t\rangle$ is the group $\langle Q, t^kg_q^{-1}\rangle$, and it is isomorphic to $Q \rtimes_{ad_{g_q} \circ \alpha^k} \langle t'\rangle$ for $t'=t^kg_q^{-1}$. 

\section{Semi-direct products}
\subsection{Setting and statement}
 
This section is for proving \cref{thm;main}.
 
\begin{theorem}\label{thm;main} If $G$ is a torsion-free hyperbolic group, and $\alpha$ is an automorphism of $G$, then $G\rtimes_\alpha \mathbb{Z}$ is relatively hyperbolic with respect to the suspensions of the polynomially growing subgroups for $\alpha$.
\end{theorem}
 
In the following, $G$ is a torsion-free hyperbolic group, and $\alpha$ is an automorphism. Let $k\geq 1$ be an integer. Since $G\rtimes_{\alpha^k} \bbZ$ embeds as a subgroup of index $k$ in the group $G\rtimes_{\alpha} \bbZ$, and as relative hyperbolicity is preserved by passing to and from a finite index subgroup (see \cite[Theorem 5.7]{drutu} or \cref{thm;qi_invariance_drutu}), we will freely use a power of $\alpha$ when needed.

\subsubsection{The telescopic argument}\label{sec;telescopic}
 
Before getting into the proof, recall that by adapting a result of Osin \cite[Theorem 2.40]{osin} (the adaptation required is to substitute ordinary Dehn functions in the statement by relative Dehn functions), we have that if a group $G$ is relatively hyperbolic with respect to the conjugates of subgroups $P_i$, and if each $P_i$ is relatively hyperbolic with respect to conjugates of subgroups $Q_{i,j}$, then $G$ is relatively hyperbolic with respect to conjugates of the $Q_{i,j}$. We will use it together with the induction hypothesis. We note that the above result is also obtained by using the asymptotic cone characterization of Drutu-Sapir \cite{drutu_sapir}.

 \subsubsection{Induction on the Kurosh rank}
We aim to prove \cref{thm;main}. We proceed by an overall induction on the Kurosh rank of $G$, for the Grushko free factor system.

We first prove the result if $G$ is of Kurosh rank $1$. In that case $G$ is either cyclic (and there is nothing to show), or freely indecomposable and torsion-free, and thus one-ended. 
Then, we will treat the case of larger Kurosh rank, for the Grushko free factor system. We pass the problem from $G$ to polynomially growing subgroups in the metric of a tree for which $\alpha$ is fully irreducible.   If this latter tree has more than $2$ orbits of edges, we will use the telescopic argument, together with \cite[Proposition 1.13]{dahmani_li}. If it has only one orbit of edges, actually the telescopic argument does not allow to use induction since $G$ itself is polynomially growing in the metric of the tree, and we treat this case separately by analysing the structure closely.  The latter case also includes Kurosh rank $2$, except when $G$ is a free group of rank $2$.
 
\subsection{The one-ended case}
We first treat the one-ended case. This case was probably known by folklore, and appears in \cite{Gautero_Lustig_one-ended}. We briefly propose a way to cover it.
 
We use the pA-tree $T_{pA}$ for $G$ and $\alpha$ (see \cref{sec;pA}).  The group $G\rtimes_{\alpha^k} \mathbb{Z}$  acts on $T_{pA}$, and is thus decomposed as a graph of groups, with abelian edge groups, and vertex groups that are suspensions of the vertex stabilizers in $G$ by the induced automorphism. The suspensions of the pA-vertex groups are hyperbolic relative to the free abelian groups of rank $2$ corresponding to the suspensions of their boundary components. The acylindrical combination theorem \cite{dahmani_combination} (\cref{thm;combination}) allows to obtain the relative hyperbolicity of $G\rtimes_{\alpha^k} \mathbb{Z}$ with respect to the groups obtained as graphs-of-groups from the connected components of $G\backslash (T_{pA} \setminus V_{pA}) $. Those are easily seen to be polynomially growing subgroups. This proves the result in this case, for the automorphism $\alpha^k$. Since $G\rtimes_{\alpha^k} \mathbb{Z}$ is of finite index in $G\rtimes_{\alpha} \mathbb{Z}$, we also have the result for the latter, by \cref{prop;fi_polynomial_growth}.
 
We proved the result for the Kurosh rank equal to $1$, we may proceed and assume that it holds for all Kurosh ranks less than that of $G$.

\subsection{The general case of full irreducibility}\label{sec;general_case}

\subsubsection{Train tracks for $\alpha$}

Recall that, if $G$ is not one-ended, it admits a proper free factor system $\calH_m$ that is $\alpha$-invariant, and for which  $\alpha$ is fully irreducible, see the discussion  in \cref{sec;ffs_fi}.  

Recall that if $T$ is a $G$-tree in which the elliptic subgroups form $\calH_m$, a map $f:T\to T$ realizes $\alpha$ if for every vertex $v\in T$, and all $g\in G$, $f(gv)= \alpha(g)f(v)$. Such a map defines equivalence classes on the link of each vertex: two edges $e_1, e_2$ starting at $v$ are equivalent if $f(e_1)$ and $f(e_2)$ have a common initial edge. A turn is a pair of edges sharing a vertex. A turn is legal if the edges are in different equivalence classes. A path is legal if it contains only legal turns. The map $f$ is a train track map if it sends edges on legal paths, and legal turns on legal turns.
 
Francaviglia and Martino construct in \cite{francaviglia_martino} a $(G,\calH_m)$-tree $T$ whose elliptic subgroups are exactly the groups in $\calH_m$, and a map $f:T\to T$ realizing $\alpha$ and that is a train track map.

Moreover, by choosing correctly the metric on $T$, one can show that $f$ stretches every edge of $T$ by the same factor \cite[Lemma 8.16]{francaviglia_martino}, that this factor is stricly larger than $1$ if $T$ contains at least two $G$-orbits of edges \cite[Lemma 1.11]{dahmani_li}, and that, at each vertex of $T$, there is at least one legal turn (see \cite[Remark 6.5]{francaviglia_martino} or the comment before \cite[Definition 8.10]{francaviglia_martino}).
 
\begin{lemma}\label{lem;some_exp} Assume that $T$ has at least two $G$-orbits of edges, then there is at least one element in $G$ that has exponential $\calH_m$-growth. \end{lemma}
 
\begin{proof}
Since the stretching factor is strictly larger than $1$, it suffices to find an element $g$ and a point $v$ in $T$ such that the segment $[v,gv]$ is legal (so that its images by the train track map grow precisely as $\lambda^n$ with $\lambda>1$). Start from a vertex $w$ and select an edge $e_1$ starting at $w$, ending at $w_1$. We observe that, by a property that we recorded before the statement, there is an edge starting at $w_1$ that makes a legal turn with $e_1$. Assume one has a path $e_1 \dots e_k$ making legal turns, one can continue and find $e_{k+1}$ also making a legal turn with $e_k$. Eventually, the edge $e_n$ will be in the same orbit as an earlier edge $e_m$, for $m<n$.  Then call $v$ the initial point of $e_m$, and $g$ the (unique) element such that $ge_m=e_n$.  We thus found a legal path as required, this proves the lemma. \end{proof}

\subsubsection{The case of large Scott complexity for $\calH_m$: hyperbolicity relative to polynomial $\calH_m$-growth}
 
Recall that the Scott complexity of the decomposition corresponding to $\calH_m$ is the quantity $(r,m)$, where $r$ is the rank of the free group. Small Scott complexity corresponds to $(0,2)$ and $(1,1)$, which are the cases when the corresponding Bass-Serre trees have exactly one orbit of edges.
 
Invoking Corollary 2.3 of \cite{dahmani_li} we know that $G\rtimes \langle t\rangle$ is relatively hyperbolic with respect to the suspensions of the maximal subgroups of polynomial $\calH_m$-growth (recall that by $\calH_m$-growth we mean growth in the tree-metric of a $(G,\calH_m)$-tree, see \cref{lem;H-growth}).  By \cref{lem;some_exp}, we are in the case that $G$ itself is not   polynomially $\calH_m$-growing for $\alpha$.  By \cite[Prop. 1.13]{dahmani_li}, if $Q$ is a maximal subgroup of $G$ that has polynomial $\calH_m$-growth for $\alpha$,  it is hyperbolic and of strictly smaller Kurosh rank than $G$.  We may therefore apply the induction assumption to the   polynomially $\calH_m$-growing  subgroups for $\alpha$.

Clearly a word-polynomially growing subgroup for $\alpha$ is a subgroup of a group of polynomial $\calH_m$-growth for $\alpha$. As a consequence, by the telescopic argument of \cref{sec;telescopic}, we have that, assuming the induction hypothesis, in the case that $T$ has more than two orbits of edges, the semi-direct product $G\rtimes_\alpha \langle t\rangle $ is relatively hyperbolic with respect to the suspensions of its polynomially growing subgroups.
 
\subsubsection{A case of small Scott complexity: the semidirect product of a preserved free product $H_1*H_2$}
 
First consider the case where $G=H_1*H_2$ and $H_1$ and $H_2$ are non-trivial, preserved by $\alpha$. {Note that the Scott complexity of this decomposition is $(0,2)$.} Then, $G\rtimes_\alpha \langle t \rangle$ is of the form $(H_1\rtimes \langle t_1 \rangle ) *_{t_1=t_2} (H_2\rtimes \langle t_2 \rangle )$, in which $t_i$ realizes $\alpha|_{H_i}$.
 
Since the Kurosh ranks of $H_1$ and $H_2$ are strictly less than that of $G$, we may apply the induction assumption to them. Their semidirect products are assumed to be relatively hyperbolic with respect to the polynomially growing subgroups for $\alpha$.
 
We discuss whether $t_i$ are parabolic or not in the groups $(H_1\rtimes \langle t_1 \rangle )$ and $(H_2\rtimes \langle t_2 \rangle )$. It is worth noting that since $t_i$ generates the cofactor $\bbZ$ of its semidirect product, it generates a maximal cyclic subgroup of $(H_i\rtimes \langle t_i \rangle )$.

We also observe that $t_i$ is parabolic in $H_i \rtimes \langle t_i \rangle$ if and only if $\alpha$ preserves a maximal polynomially $\|\cdot\|_w$-growing subgroup, since every parabolic subgroup is the suspension of a maximal polynomially $\|\cdot\|_w$-growing subgroup. 

If both $t_i$ are loxodromic in the relatively hyperbolic groups $(H_i\rtimes \langle t_i \rangle )$, the amalgam is relatively hyperbolic with respect to the family of conjugates of the parabolic subgroups of the factors (see \cref{thm;combination}). This proves the desired result. If one is parabolic, and not the other, (say $t_1$ is parabolic while $t_2$ is loxodromic) we may enrich the peripheral structure of $H_2 \rtimes \langle t_2 \rangle$ by adding the conjugates of this maximal cyclic subgroup $ \langle t_2 \rangle$, and it remains relatively hyperbolic (by Yang's peripheral extension theorem, \cref{theo;yang}). Indeed, the cyclic subgroup generated by $t_2$ is relatively quasiconvex, by \cite[Corollary 4.20]{osin}, and is malnormal because it is maximal cyclic. 
The action of $G \rtimes_{\alpha} \mathbb{Z}$ on the Bass-Serre tree of $(H_1\rtimes \langle t_1 \rangle ) *_{t_1=t_2} (H_2\rtimes \langle t_2 \rangle )$ is 2-acylindrical because if two different edges are adjacent to the vertex fixed by $(H_2\rtimes \langle t_2 \rangle )$ their stabilizers are different conjugates of  $\langle t_2 \rangle $ in this group, hence have trivial intersection by malnormality of $\langle t_2 \rangle $. Therefore the combination theorem (\cref{thm;combination}(1)) of \cite{dahmani_combination} still applies.  
 
If both $t_i$ are parabolic in $(H_i\rtimes \langle t_i \rangle )$, let $P_1$, $P_2$ be their respective maximal parabolic subgroups. In the notation of \cref{sec;free_product}, we have $P_1= A\rtimes \langle t_1\rangle$ and $P_2= B \rtimes \langle t_2\rangle$, but for notation purpose, we will write $A=Q_1$ and $B=Q_2$.  We may write \[(H_1\rtimes \langle t_1 \rangle ) *_{t_1=t_2} (H_2\rtimes \langle t_2 \rangle )\] as \[ \Big((H_1\rtimes \langle t_1 \rangle )\Asterisk_{P_1} [ P_1 *_{t_1=t_2} P_2] \Big) \Asterisk_{P_2} (H_2\rtimes \langle t_2 \rangle ). \]
 
Once again, the combination theorem \cite{dahmani_combination} (\cref{thm;combination} (2) and (3)) applies at every step of the combination, and one obtains that $G\rtimes_\alpha \bbZ$ is hyperbolic relative to the conjugates of the suspensions of the former polynomially growing subgroups of $H_1$ and $H_2$ and also the conjugates of $P_1 *_{t_1=t_2} P_2$.

It remains to see that $P_1 *_{t_1=t_2} P_2$ is a suspension of a polynomially growing subgroup of $G$ for $\alpha$. As mentioned,  $P_i = Q_i \rtimes \langle t_i \rangle$, and we know from \cref{sec;free_product} that $Q_1*Q_2$ is a maximal polynomially growing subgroup of $G$ for $\alpha$, and each $Q_i$ is preserved by $\alpha$. This describes   $P_1 *_{t_1=t_2} P_2$ as $(Q_1*Q_2) \ltimes \langle t_1\rangle$, thus as a suspension of a maximal polynomially growing subgroup.

\subsubsection{Another case of small Scott complexity: the semi-direct product of a free product $H*\bbZ$ with $H$ preserved.}

We do the same job for the case $G= H*\langle s \rangle$ in which $\alpha(H)=H$, which is the last case we need to consider.
 
Observe that we can assume that $\alpha(s) = sh^{-1}$ for some $h\in H$. We can then write $(H*\langle s \rangle) \rtimes_\alpha \langle t \rangle$ as 
\[(H*\langle s \rangle) \rtimes_\alpha \langle t \rangle = (H\rtimes \langle t \rangle) *^{(s)}_{\langle t\rangle, \langle th\rangle},\] meaning that in the rightmost HNN, the stable letter is $s$ and it conjugates $t$ to $th$.
 
The group $H\rtimes \langle t \rangle$ is, by induction hypothesis, relatively hyperbolic with respect to polynomially growing subgroups in $H$ for $\alpha$.
 
Again, the group generated by $t$ and $th$ are maximal cyclic groups in $H\rtimes \langle t \rangle$, since they generate the cofactor $\bbZ$.
And again we discuss according to the parabolicity of the elements $t$ and $th$ in $(H\rtimes \langle t \rangle)$. Recall from \cref{sec;freehnn}  that $A_0$ is the (possibly trivial, but unique) maximal $\|\cdot\|_w$-polynomially growing subgroup of $H$ that is fixed by $\alpha$, while $B_0$ is the (possibly trivial, but unique) maximal $\|\cdot\|_w$-polynomially growing subgroup such that $\alpha(B_0) = hB_0h^{-1}$. Recall also from \cref{sec;freehnn} that $K$ is the set of solutions of the membership equation $k^{-1}h^{-1}\alpha(k) \in A_0$ (of unknown $k$).

By induction, there is a relatively hyperbolic structure on $H\rtimes \langle t\rangle$ with peripheral subgroups the suspensions of the polynomially growing subgroups of $H$. Note that in this structure, $t$ is parabolic if and only if $A_0 \neq 1$, and $th$ is parabolic if and only if $B_0 \neq 1$. However, even when $A_0 = 1$, if $ K\neq \emptyset$, then by \cref{lem;Kcoset}, we have $K = \{k_0\}$, and $t$ normalizes the polynomially growing subgroup generated by $sk_0$ (which is not contained in $H$, see case 2 below).
By a result of Osin \cite[Corollary 4.20]{osin}, $\langle t\rangle$ is relatively quasiconvex in $H \rtimes \langle t\rangle$. The subgroup $\langle t\rangle$ is also malnormal, as otherwise $A_0$ would be nontrivial. 
We will therefore add the cyclic subgroup $\langle t\rangle$ to the peripheral structure of $H\rtimes \langle t\rangle$. The hypotheses of Theorem 1.1 of \cite{yang_ext_peripheral} (\cref{theo;yang}) are thus satisfied by this new peripheral structure, and  $H\rtimes \langle t\rangle$ with this new peripheral structure is still a relatively hyperbolic group.

In order to prove the theorem, we have to analyze four cases here (note that by \cref{coro;A_0B_0}, those are the only possible cases). 

\begin{enumerate}
    \item $A_0$ and $B_0$ are trivial, and $K=\emptyset$.
    \item either $A_0$ and $B_0$ are trivial, and $K=\{k_0\}$, or $A_0$ is non-trivial while $B_0$ is trivial and $K = \emptyset$, or 
 $A_0$ is trivial and $B_0$ is non-trivial, and $K=\emptyset$. 
    \item $A_0$ and $B_0$ are nontrivial and $K = \emptyset$.
    \item $A_0$ and $B_0$ are nontrivial and $K=\{k_0A_0\}$.
\end{enumerate}

The first case above is when both $t$ and $th$ are loxodromic. Here, the acylindrical combination theorem of \cite{dahmani_combination} (\cref{thm;combination} (1)) gives the conclusion. Indeed, $G \rtimes \langle t \rangle$ acts acylindrically on the Bass-Serre tree of the HNN extension $(H\rtimes \langle t \rangle) *^{(s)}_{\langle t\rangle, \langle th\rangle}$ of $H \rtimes \langle t\rangle$ since $G = H * \langle s\rangle$.

If one of them is loxodromic and the other is parabolic (case 2  above), also expanding the peripheral structure (as in the previous case of the free product of two invariant factors) we can still use the acylindrical combination theorem. 

Cases 3 and 4 correspond to both $t$ and $th$ being parabolic.  

Case 3, which is again similar to that of the previous subsection, is when $t$ and $th$ belong to two different parabolic subgroups, $A_0 \rtimes \langle t \rangle$ and $B_0 \rtimes \langle th \rangle$.  
Manipulation of presentations shows that, if $\dot B_0$ is another abstract copy of $B_0$,
\[ \begin{array}{ccl} 
(H*\langle s \rangle) \rtimes_\alpha \langle t \rangle & = & \left\langle H, s, t, \dot \tau, \dot B_0 \; |\; \begin{array}{ccc} \dot \tau & =& t, \\ \dot{\tau}^s & =& th, \\ {\dot B_0}^s & \equiv & B_0,\\ H^t & \equiv & \alpha(H).\end{array} \right\rangle \\ 
& = & \left[ \left( H\rtimes \langle t \rangle \right) \, \Asterisk_{A_0\rtimes \langle t\rangle} \left( \langle A_0, t\rangle *_{t=\dot \tau} \langle \dot B_0, \dot \tau \rangle \right) \right] \Asterisk^{(s)}_{\langle th, B_0 \rangle, \langle \dot \tau, \dot B_0 \rangle} \end{array}       \]

See also \cref{fig:moves} for the topological meaning of these identifications.

\begin{figure}
\centering
\includegraphics[width=\textwidth]{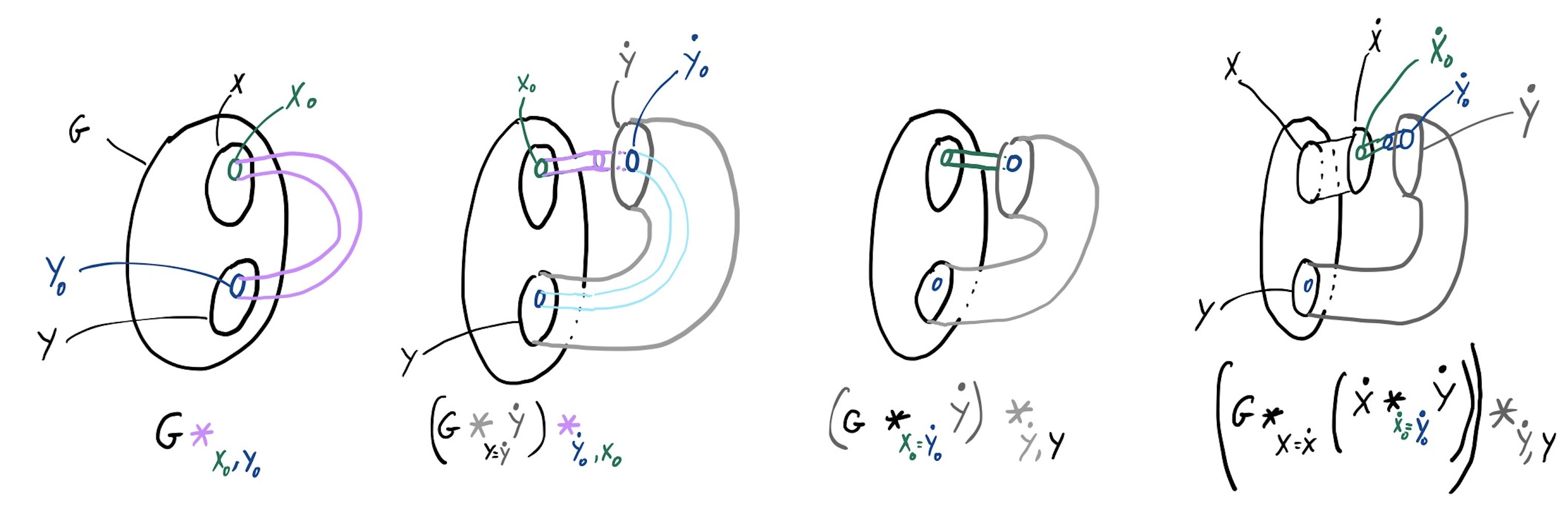}
\caption{Filling the subgroups of an HNN extension over some subgroups (the pictures show spaces of which one takes the fundamental group; the group is the same for each of the four pictures)}\label{fig:moves}
\end{figure}

In the last decomposition, one can apply the combination theorem of \cite{dahmani_combination} (\cref{thm;combination}) to each of the constructions, in turn, to obtain that $(H*\langle s \rangle) \rtimes_\alpha \langle t \rangle $ is relatively hyperbolic with respect to the conjugates of the parabolic subgroups of $H\rtimes \langle t \rangle$ except the two classes $A_0 \rtimes \langle t \rangle $ and $B_0 \rtimes \langle th \rangle $, but with in addition, the conjugacy class of $\left( \langle A_0, t\rangle *_{t=\dot \tau} \langle \dot B_0, \dot \tau \rangle \right)$ which is $\left( \langle A_0, t\rangle *_{t} \langle (B_0)^{s^{-1}}, t \rangle \right)$. Observe that $t$ normalizes both $A_0$ and $(B_0)^{s^{-1}}$, and that those two groups are polynomially growing for $\alpha$. It follows that $\left( \langle A_0, t\rangle *_{t=\dot \tau} \langle (B_0)^{s^{-1}}, t \rangle \right)  =   (A_0 * (B_0)^{s^{-1}}) \rtimes \langle t \rangle$. 
We thus recognize the suspension of a maximal polynomially growing subgroup, as constructed.

Assume now that $K$ is non-empty, and let $k_0\in K$ (case 4 above). Note that $t$ and $th$ are in different parabolic subgroups of $H \rtimes \langle t \rangle$ if and only if $B_0 \neq A_0$. If $B_0 = A_0$, we may take $k_0 = 1$. We may complete the previous calculation as 
 \[ \begin{array}{ccl} 
(H*\langle s \rangle) \rtimes_\alpha \langle t \rangle  =\left[ \left( H\rtimes \langle t \rangle \right) \, \Asterisk_{A_0\rtimes \langle t\rangle} \left( \langle A_0, t\rangle *_{t=\dot \tau} \langle \dot B_0, \dot \tau \rangle \right) \right] \Asterisk^{(sk_0)}_{\langle th, B_0 \rangle^{k_0}, \langle \dot \tau, \dot B_0 \rangle}.  \end{array}       \]

We thus see that the HNN extension is extending a peripheral subgroup of $H \rtimes_\alpha \langle t \rangle$, since $sk_0$ is in the same maximal polynomially growing subgroup as $\langle A_0, t\rangle$ (as seen in \cref{sec;freehnn}).

The group is therefore relatively hyperbolic with respect to  the conjugates of the parabolic subgroups of $H\rtimes \langle t \rangle$ except the two classes $A_0 \rtimes \langle t \rangle $ and $B_0 \rtimes \langle th \rangle $, but with in addition, the conjugacy class of $\left\langle \left( \langle A_0, t\rangle *_{t=\dot \tau} \langle \dot B_0, \dot \tau \rangle \right), sk_0\right\rangle$, which is $\left(\langle A_0, sk_0, t  \rangle   *_{t} \langle (B_0)^{s^{-1}}, t \rangle \right)$. Recall that by \cref{coro;A_0B_0}, one has $B_0^{s^{-1}} = A_0^{ (sk_0)^{-1}}$, so  the last parabolic subgroup is   $\langle A_0, sk_0, t  \rangle$. Since $t$ normalizes $\langle A_0, sk_0\rangle $, we have that this new peripheral subgroup is a suspension of a maximal polynomially growing subgroup of $H*\langle s\rangle$. 
 
This finishes the induction, and proves the theorem.

\bibliographystyle{alpha}
\bibliography{Relative_hyperbolicity_hyperbolic_by_cyclic_groups}
\end{document}